\documentclass[11pt]{amsart}
\usepackage{fullpage,hyperref}
\usepackage{tikz-cd}
\usepackage{color,soul}
\usepackage{amssymb}
\usepackage[capitalise, nameinlink]{cleveref}
\usepackage[all]{xy}
\usepackage{enumitem}

\makeatletter
\@namedef{subjclassname@2010}{%
  \textup{2010} Mathematics Subject Classification}
\makeatother
\newtheorem{theorem}{Theorem}[section]
\newtheorem{prop}[theorem]{Proposition}
\newtheorem{cor}[theorem]{Corollary}
\newtheorem{lem}[theorem]{Lemma}
\theoremstyle{definition}
\newtheorem{defn}[theorem]{Definition}
\newtheorem{rem}[theorem]{Remark}

\newtheorem{pgr}[theorem]{}

\newtheorem{qst}[theorem]{Question}
\newtheorem{ntn}[theorem]{Notation}
\newcommand{\bset}[1]{\bigl\{ #1 \bigr\}}

\newcounter{IntroCount}

\newtheorem{mainthm}[IntroCount]{Theorem}
\newtheorem{maincor}[IntroCount]{Corollary}

\newcommand{\ep}{\varepsilon}

\newcommand{\ld}{\lambda}

\newcommand{\ph}{\varphi}
\newcommand{\ps}{\psi}
\newcommand{\rh}{\rho}

\newcounter{cuax}

\renewcommand{\thecuax}{O\arabic{cuax}}

\def\N{{\mathbb{N}}}

\def\k{{\mathcal{K}}}

\newcommand{\hm}{homomorphism}
\renewcommand{\k}{\mathcal{K}}

\newcommand{\ca}[1]{$\mathrm{C}^*$-algebra#1}
\newcommand{\Cu}{\mathrm{Cu}}

\hypersetup{
	colorlinks,
	linkcolor=blue,          
	filecolor=magenta,      
	urlcolor=cyan           
}

\title{Continuous functions over a pure C*-algebra}

\author{Apurva Seth, Eduard Vilalta}
\address{Apurva~Seth,
Mathematical Institute, University of Oxford, Radcliffe Observatory, Andrew Wiles Building, Woodstock Rd, Oxford OX2 6GG, United Kingdom.}
\email{apurvaseth14@gmail.com}
\address{Eduard~Vilalta, 
Department de Matem\`{a}tiques, Universitat Polit\`{e}cnica de Catalunya, Diagonal 647, Barcelona, Spain.}
\email{eduard.vilalta@upc.edu}
\urladdr{www.eduardvilalta.com}

\subjclass[2020]%
{Primary
46L05. 
}

\date{\today}

\thanks{AS was supported by the Engineering and Physical Sciences Research Council (EP/X026647/1) and EV was partially supported by the Spanish State Research Agency (grant No. PID2023-147110NB-I00). For the purpose of open access, the authors have applied a CC BY public copyright license to any author-accepted manuscript arising from this submission.}

\begin{document}
\begin{abstract}
    Let $X$ be a compact metric space, and let $A$ be a pure \ca{}. We show that $C(X,A)$ is pure whenever
    \begin{enumerate}
        \item $A$ is simple; or 
        \item every quotient of $A$ is stably finite (e.g., $A$ has stable rank one).
    \end{enumerate}

    Using permanence properties of pureness, we prove that the tensor product of any such $A$ with any ASH-algebra is pure.
\end{abstract}
\maketitle

\section{Introduction}

The properties of \emph{strict comparison} and \emph{almost divisibility} play a fundamental role in the modern structure theory of \ca{s}. As defined by Winter \cite{Win12NuclDimZstable}, \ca{s} that satisfy both properties are termed \emph{pure}. In the simple finite setting, such algebras should be regarded as the correct analogues of II$_1$ factors; roughly speaking, the defining properties of pureness correspond to two key features of projections in II$_1$factors, adapted to the setting of positive elements ---a necessary adaptation given that many \ca{s} contain few or no nontrivial projections.

Just as the unique trace on a II$_1$ factor $\mathcal{N}$ determines the order of projections, \emph{strict comparison} \cite{Bla88Comparison} asserts that tracial state data determines when one positive operator can be approximately compressed to another (Remark \ref{rem:StCompAlmUnperf}). Similarly, \emph{almost divisibility} (Definition \ref{def_almost_divisibility}) serves as the analogue to the fact that projections in $\mathcal{N}$ can be divided arbitrarily. Whilst these two properties are automatic for II$_1$ factors, they can fail for simple \ca{s} \cite{Tom08ClassificationNuclear,Vil99SRSimpleCa}. Consequently, determining which algebras satisfy these properties ---and, in particular, developing mechanisms to deduce pureness--- has been a major focus in the field for the past few years \cite{AGKEP25,AntPerThiVil24arX:PureCAlgs,Rob25Self}. Examples of pure \ca{s} can be found both in the nuclear and non nuclear case, and include $\mathcal{Z}$-stable algebras, purely infinite algebras \cite{kirchberg2000non}, von Neumann algebras with trivial type I summand, and reduced (twisted) group \ca{s} of acylindrically hyperbolic groups \cite{FloLisCobPag:PureTwi,RauThiVil:StCompTwi}.

In the realm of simple nuclear \ca{s}, pureness features prominently in the Toms-Winter conjecture \cite{CasEviTikWhiWin21NucDimSimple,Ror04StableRealRankZ,Win12NuclDimZstable} and serves as a crucial regularity condition in the Elliott classification program \cite{Whi23ICM,Win18ICM}. More generally, the notion is expected to play an important role in the classification of $^*$-homomorphisms \cite{CarGabSchTikWhi23arX:ClassifHom1,Sza26:UniqThmKK}, potentially by weakening the requirement of $\mathcal{Z}$-stability on the codomain to that of pureness. In fact, as shown in \cite[Theorem~7.3.11]{antoine2018tensor}, one may think of pureness as $\mathcal{Z}$-stability at the level of the Cuntz semigroup.

Of the two defining properties of pureness, strict comparison has received the most attention in recent years, finding deep connections with the theory of group \ca{s} \cite{AGKEP25,Oza25:Proxim}, time-frequency analysis \cite{BedEnsvVelt22,EnsThiVil25Criteria}, and the recent negative resolution to Tarski's problem \cite{KES:NegativeTarski}. However, despite its importance, strict comparison is often difficult to verify directly. For this reason, attention has shifted towards the stronger, more rigid property of pureness, which is particularly susceptible to \emph{dimension reduction phenomena}. Specifically, it was recently proved in \cite{AntPerThiVil24arX:PureCAlgs} that any $(m, n)$-pure \ca{} ---a formal `quantized' weakening of pureness--- is automatically pure. This result allows one to deduce pureness (and thus strict comparison) in situations where direct verification was previously out of reach. It has also paved the way for new permanence results; for instance, it was established in \cite{perera2025extensions} that an extension of pure \ca{s} is itself pure, mirroring the corresponding property for $\mathcal{Z}$-stability \cite[Theorem~4.3]{TomWin07ssa}. An important question in this line of inquiry, raised at a number of conferences, is whether a minimal tensor product $B \otimes A$ is pure whenever $A$ is. In the present paper, we investigate this question for the case $B = C(X)$.

\begin{mainthm}[\ref{Purehomogenous},~\ref{prp:MainNonSimp}]\label{prp:MainThm1}
    Let $X$ be a compact metric space, and let $A$ be a pure \ca{}. Assume additionally that
    \begin{enumerate}
        \item $A$ is simple; or 
        \item every quotient of $A$ is stably finite.
    \end{enumerate}

    Then, $C(X,A)$ is pure.
\end{mainthm}

Before outlining the proof strategy and the challenges involved in establishing Theorem \ref{prp:MainThm1}, we highlight some important consequences. First, combining Theorem \ref{prp:MainThm1} with several permanence properties of pureness, we obtain the following application:

\begin{mainthm}[\ref{prp:ASHtenPure}]\label{prp:MainThm2}
    Let $A$ be a pure \ca{} satisfying (1) or (2) in Theorem \ref{prp:MainThm1}, and let $B$ be a unital separable ASH-algebra. Then, $A\otimes B$ is pure.
\end{mainthm}

As a particular example of this result, Theorem \ref{prp:MainThm2} shows that the tensor product of any Villadsen algebra of type I with the reduced group \ca{} of any free group is pure. It is currently unknown whether these specific examples are $\mathcal{Z}$-stable (Question \ref{qst:VillZstab}), highlighting the utility of pureness as a regularity invariant. In fact, Theorem \ref{prp:MainThm2} is slightly more general: As we prove in Theorem \ref{DpureRSHA}, any recursive subhomogeneous \ca{} over $A$ \cite[Definition~3.2]{archey2020structure} (Definition \ref{RSHA}) is pure whenever $A$ is.

A further application of Theorem \ref{prp:MainThm2} arises in the study of determining which nonamenable groups have a reduced group \ca{} with strict comparison. Since $C_r^*(G)$ is a direct sum of simple pure \ca{s} whenever $G$ is countable and acylindrically hyperbolic \cite[Proof of Theorem~C]{FloLisCobPag:PureTwi}, and $C_r^*(H)$ is subhomogeneous whenever $H$ is virtually abelian, we get:

\begin{maincor}\label{prp:MainCor1}
 Let $G$ be a countable acylindrically hyperbolic group, and let $H$ be a virtually abelian group. Then, $C_r^*(G\times H)$ is pure. In particular, it has strict comparison.
\end{maincor}

In addition to significantly enlarging the known class of pure \ca{s}, Theorems \ref{prp:MainThm1} and \ref{prp:MainThm2} shed new light on the \emph{Global Glimm Problem}, a major open problem in the field \cite[Problem~LXXIII]{SchTikWhi99arX:Prob}. The problem asks whether the \emph{Global Glimm Property} ---a generalization of non-elementariness for non-simple \ca{s} \cite{KirRor02InfNonSimpleCalgAbsOInfty} which we recall in Paragraph \ref{pgr:GGP}--- is equivalent to having no nonzero elementary ideals of quotients, and has been solved affirmatively in numerous cases (see \cite[Section~4]{Vil25:IntroGGP} for an overview). While it was previously known that $C(X, A)$ possesses the Global Glimm Property whenever $X$ is finite-dimensional and $A$ is simple and has the property \cite[Theorem~4.3]{BlaKir04GlimmHalving}, this was not known for pure \ca{s} $A$ satisfying condition (2) of Theorem \ref{prp:MainThm1}, nor for the tensor products appearing in Theorem \ref{prp:MainThm2}. As we explain below, this is not a consequence of our results; rather, it is a necessary intermediate step in our proof of pureness.

Apart from the applications presented in this paper, Theorem \ref{prp:MainThm1} has further consequences. In particular, this result will be used in forthcoming work of the first named author with Evington, Hua, Schafhauser, and White, which will study \emph{$K$-stability} for pure \ca{s} satisfying conditions as in Theorem~\ref{prp:MainThm1}. It has been shown in \cite{lin2025strict} that finite simple pure \ca{s} have stable rank one and are therefore $K_1$-bijective. The future work will focus on the stronger property of $K$-stability for this class.

\subsection*{Strategy of the proof} Given a \ca{} $A$ with strict comparison, it is well-known that $C(X,A)$ need not have comparison. Topological obstructions, often manifesting as non-trivial vector bundles, can impede comparison even in simple settings ---for example, when $A=\mathbb{C}$ and $X=\mathbb{T}^4$ \cite{Vil98SimpleCaPerforation}. Thus, one of the challenges in establishing Theorem \ref{prp:MainThm1} is that strict comparison and almost divisibility cannot be verified separately.

Our strategy relies on a multi-step dimension reduction argument. First, we restrict our attention to the case where $X$ is finite-dimensional. In Sections \ref{sec:CompCuA} and \ref{sec:DivCuA}, we prove that if $X$ is a compact metric space of dimension $m$ and $A$ is pure, the algebra $C(X, A)$ satisfies a weakening of $(m,n)$-pureness, with $n$ depending only on $m$. Concretely, we see that $C(X,A)$ has $m$-comparison (Proposition \ref{comp_c(x,A)}) and satisfies a \emph{weak} version of $n$-almost divisibility (Corollary \ref{robert_analogy_2}); see Definitions \ref{dfn:MComp} and \ref{def_almost_divisibility} for details.

As Corollary \ref{robert_analogy_2} does not establish $n$-almost divisibility for any $n$, the rigidity results from \cite{AntPerThiVil24arX:PureCAlgs} cannot be applied directly. Instead, building on the techniques from \cite{AntPerThiVil24arX:PureCAlgs}, we prove a new general reduction result (Theorem \ref{prp:Gen_DimRed}), which we then use in Theorem \ref{prp:CXAPureIfGGP} to deduce that $C(X, A)$ is pure if and only if it possesses the Global Glimm Property. Relying on the fact that all these properties pass to inductive limits, we are able to dispense with the assumption of finite dimensionality on $X$. Thus, the proof of Theorem \ref{prp:MainThm1} reduces to proving the Global Glimm Property for $C(X,A)$. For \ca{s} satisfying condition (1), this follows from \cite[Theorem~4.3]{BlaKir04GlimmHalving}; for those satisfying (2), the result is new and is proven in Proposition \ref{prp:CXNoQuot}.

\subsection*{Acknowledgments} This paper started during EV's visit to the University of Oxford during the \emph{Mini Course: Topological Phenomena in the Cuntz semigroup}, which was given by Andrew Toms and organized by Stuart White.  We are grateful to both, as well as to the University of Oxford and St John’s College, for their kind hospitality. We would further like to thank Stuart White for his helpful feedback on the paper, and Hannes Thiel for pointing out Corollary \ref{prp:MainCor1}. AS is  grateful to N. Christopher Phillips for valuable discussions concerning pureness of \ca{s} and for feedback on earlier drafts of this paper. AS also thanks Julian Buck, Dawn Archey, and Javad Mohammadkarimi for insightful discussions.

\section{Preliminaries}
Let $A$ be a \ca{}, and let $a, b$ be positive elements in $A$. As introduced in \cite{Cun78DimFct}, we say that  $a$ is \emph{Cuntz subequivalent} to $b$, written $a \precsim b$, if there is a sequence $(v_n)_{n = 1}^{\infty}$ in $A$
such that $\lim_{n} v_n b v_n^* = a$. We say that $a$ and $b$ are {\emph{Cuntz equivalent},
written $a \sim b$, if $a \precsim b$ and $b \precsim a$.   The relation $\precsim$ is transitive and reflexive, and $\sim$ is an equivalence relation.

\begin{defn}[{\cite{Coward}}]\label{df:CuA}
Let $A$ be a \ca. The \emph{Cuntz semigroup} of $A$ is defined as 
\[\Cu(A)=(A\otimes\k)_+/\!\sim.\]
 For a positive element $a\in (A\otimes\k)_+$, we denote by $[a]$ its Cuntz equivalence class, hence $\Cu(A)=\big\{[a]\colon a\in (A\otimes\k)_+\big\}$.
There is a natural partial order defined on $\Cu(A)$, namely \[[a]\leq [b]\quad\text{if}\quad a\precsim b.\]  
Further, the diagonal sum $a\oplus b$ induces an addition on $\Cu (A)$.
\end{defn}

\begin{ntn}\label{D:MinusEp}
Let $A$ be a \ca{}, let $a \in A_{+}$, and let $\ep > 0$. We define $(a-\varepsilon)_+:=f(a)$, where $f \colon [0,\infty) \to [0,\infty)$ is the function
\[
f(\lambda)
 = (\lambda - \ep)_{+}
 = \begin{cases}
     0             & \text{if } 0 \leq \lambda \leq \ep, \\
     \lambda - \ep & \text{if } \lambda > \ep.
   \end{cases}
\]
\end{ntn}

The following lemma collects several well-known results concerning $(a-\varepsilon)_+$.

\begin{lem}\label{L:CzBasic}
Let $A$ be a \ca.
\begin{enumerate}
\item\label{L:CutDowninHer}
If $a , b \in A_+$, such that $b \precsim a$, then for every $\ep>0$ there exists a positive $c\in \overline{a A a}$ with $(b-\ep)_+ \sim c$.
\item\label{L:CzBasic:MinIter}
Let $a \in A_{+}$ and let $\ep_1, \ep_2 > 0$.
Then
\[
\big( ( a - \ep_1)_{+} - \ep_2 \big)_{+}
 = \big( a - ( \ep_1 + \ep_2 ) \big)_{+}.
\]
\item\label{L:CuntzDirectSum}
Let $a, b \in A_+$,
and let $\ep_1, \ep_2 \geq 0$.
Then
\[
\big( a + b - (\ep_1 + \ep_2) \big)_{+}
   \precsim (a - \ep_1)_{+} + (b - \ep_2)_{+}
   \precsim (a - \ep_1)_{+} \oplus (b - \ep_2)_{+}.
\]
\item \label{L:Cuntzcutcutdown}
Let $\ep > 0$ and $\ld \geq 0$.
Let $a, b \in A$ satisfy $\| a - b \| < \ep$.
Then $(a - \ld - \ep)_{+} \precsim (b - \ld)_{+}$.
\end{enumerate}
\end{lem}
\begin{proof}
Part~(\ref{L:CutDowninHer}) is proved in \cite[Proposition~2.7]{kirchberg2000non}. 
Part~(\ref{L:CzBasic:MinIter}) appears in \cite[Lemma~2.5~(i)]{kirchberg2000non}. Parts~(\ref{L:CuntzDirectSum}) and (\ref{L:Cuntzcutcutdown}) appear as \cite[Lemma~1.5, Corollary~1.6]{phillips2014large}.
\end{proof}

A fundamental relation between elements in the Cuntz semigroup is that of \emph{compact containment}. Although this is generally defined in the context of \emph{abstract Cuntz semigroups} (known as \emph{$\Cu$-semigroups}), in this paper we restrict our attention to the \ca{ic} case and refer the interested reader to \cite{gardella2024modern} for a general exposition.

\begin{ntn}
    Let $A$ be a \ca{}, and let $x,y\in \Cu (A)$. We say that $x$ is \emph{compactly contained} in $y$, and write $x\ll y$, if $x=[a]$, $y=[b]$ and $a\precsim (b-\varepsilon)_+$ for some $\varepsilon>0$.

    An element $x\in\Cu (A)$ is said to be \emph{compact} if $x\ll x$.
\end{ntn}

\section{\texorpdfstring{Comparison in $\Cu(C(X,A))$}{Comparison in Cu(C(X,A))}}\label{sec:CompCuA}

In this section, we begin our investigation of the regularity properties of the Cuntz semigroup of $C(X, A)$. It is a well-known phenomenon in the commutative setting that the topological dimension of the underlying space $X$ governs the comparison properties of the algebra. Here, we extend this result to the $A$-valued case. Specifically, we show in Proposition \ref{comp_c(x,A)} that if $X$ is a finite-dimensional space of dimension $m$ and $A$ has strict comparison, then $C(X, A)$ has $m$-comparison. We start by recalling the definition of $m$-comparison and its relation to strict comparison; see Definition \ref{dfn:MComp} and Remark \ref{rem:StCompAlmUnperf} below.

\begin{ntn}\label{ntn:s_below}
    For $x, y$ elements of a Cuntz semigroup $\Cu (A)$, let us write $x \leq_s y$, if $(m + 1)\,x \leq m y$ for some $m\in\mathbb{N}$.
\end{ntn}

\begin{defn}[{\cite[Definition~2.1]{Win12NuclDimZstable}}]\label{dfn:MComp}
    Let $A$ be a \ca{} and let $n\in\N$. We say that $A$ (and its Cuntz semigroup $\Cu (A)$) has the \emph{$n$-comparison} property if $x \leq_s y_j$ for $x, y_j \in \Cu (A)$ and $j = 0, 1, \ldots, n$, implies $x\leq \sum_{j=0}^n y_j$.

    $\Cu (A)$ is said to be \emph{almost unperforated} if it has $0$-comparison.
\end{defn}

\begin{rem}\label{rem:StCompAlmUnperf}
    Let $A$ be a \ca{}. For any $[0,\infty]$-valued lower semi-continuous $2$-quasitrace $\tau$ on $A$ (which extends to $A\otimes\mathcal{K}$ and is denoted by the same symbol), the associated \emph{dimension function} $d_\tau\colon \Cu (A)\to [0,\infty]$ is defined by $d_\tau ([a])=\lim_n\tau (a^{1/n})$.

    As shown in \cite[Proposition~6.2]{elliott2011cone}, $\Cu(A)$ is almost unperforated (in the sense described above) if and only if, for every pair $a,b\in  (A\otimes\mathcal{K})_+$ such that $a$ is in the ideal generated by $b$ and $d_\tau ([a])<d_\tau([b])$ whenever $d_\tau ([b])=1$, one has $[a]\leq [b]$ in $\Cu (A)$.

    This property is known as \emph{strict comparison} (of positive elements by $[0,\infty]$-valued lower semi-continuous $2$-quasitraces). Note that, if $A$ is simple, $A$ has strict comparison if and only if, for every pair of non-zero elements $a,b\in  (A\otimes\mathcal{K})_+$ such that $d_\tau ([a])<d_\tau([b])$ for every normalized quasitrace $\tau$, one has $[a]\leq [b]$ in $\Cu (A)$. We refer the reader to \cite[Paragraph~3.4]{AntPerThiVil24arX:PureCAlgs} for an in-depth introduction.
\end{rem}

We next consider how $n$-comparison behaves under tensoring with a commutative \ca{}. In particular, tensoring with $C(X)$ weakens the comparison property in a way that reflects the topological dimension of $X$. First, we need the following lemma which is a slight modification of \cite[Propostion 3.6]{asadi2021radius}.

\begin{lem}\label{lem:CuntzSubC(X,A)}
    Let $X$ be a compact metric space of covering dimension $m$, and let $A$ be a \ca{}. Let $f,g_0,\ldots ,g_m\in C(X,A)_+$ be such that $f(x)\precsim g_k(x)$ for each $x\in X$ and $k= 0,1,\ldots,m$. Then, \[f\precsim g_0\oplus\ldots\oplus g_m.\]

    In particular, if $f,g\in C(X,A)_+$ are such that $f(x)\precsim g(x)$ for each $x\in X$, we have \[f\precsim g\otimes 1_{m+1}.\]
\end{lem}
\begin{proof}
    Let $\varepsilon>0$. By definition, for $0\leq k\leq m$, there exists $v_x^{(k)}\in A$ such that 
\begin{equation*}
\begin{split}
   \|v_x^{(k)}\, g_k(x)\,(v_x^{(k)})^*-f(x)\| &< \frac{\varepsilon}{m+1}.
   \end{split}
   \end{equation*}
   By continuity of the maps $\zeta_x^{(k)}\colon X\rightarrow [0,\infty)$, defined by
   \begin{equation*}
       \begin{split}
           \zeta_x^{(k)}(z) &=\|v_x^{(k)}\,g_k(z)\,(v_x^{(k)})^*-f(z)\|,
       \end{split}
   \end{equation*}
   we find an open neighborhood of $x$, say $N_x$, such that  for all $ z\in N_x$ and for all $k\in\{0,1,\ldots,m\}$
   \begin{equation*}
       \begin{split}
           \|v_x^{(k)} \, g_k(z)\, (v_x^{(k)})^*-f(z)\|<\frac{\varepsilon}{m+1}.
       \end{split}
   \end{equation*}
   Since $X$ is compact there exist $x_1,x_2,\ldots,x_l\in X$ such that $X=\bigcup_{j=1}^l{N_{x_j}}$. Using the fact that $\dim(X)=m$,  there exists an
$m$-decomposable finite open refinement of $\{N_{x_j}: j = 1,2,\ldots,l\}$ of the form
   
  \[\{{U_{0,j}}\}_{j\in J_0}\cup \{{U_{1,j}}\}_{j\in J_1} \cup \ldots \cup \{{U_{m,j}}\}_{j\in J_m}
  \]
  such that 
\[
U_{k,j} \cap U_{k,j'} = \emptyset, \quad\text{for}\,\, j\neq j'\quad\text{and}\quad k=0,1,\ldots,m. 
\]
Choose a partition of unity subordinate to the above cover, say $h_{k,j} \colon X \rightarrow [0,1]$, such that 
\[
\text{supp}(h_{k,j}) \subseteq U_{k,j},\quad\text{and}\quad \sum_{k=0}^m \sum_{j\in J_{k}} h_{k,j}=1.
\]

For each $k \in \{0,1,\ldots,m\}$, and $j \in J_k$, choose $\ell(k,j) \in \{1,2,\ldots,l\}$ such that $U_{k,j} \subseteq N_{x_{\ell(k,j)}}$. Define
\begin{equation*}
\begin{split}
g &:= \text{diag}\left(g_0, g_1, \dots, g_m\right),\quad\text{and}\\[2mm]
u &:=\text{diag}\left(\sum_{j\in J_0} h_{0,j}^{1/2}\, v_{x_{\ell(0,j)}}^{(0)}, \sum_{j\in J_1} h_{1,j}^{1/2}\, v_{x_{\ell(1,j)}}^{(1)}, \ldots, \sum_{j\in J_m} h_{m,j}^{1/2}\, v_{x_{\ell(m,j)}}^{(m)}\right).
\end{split}
\end{equation*}
Let $z\in X$, and $\Gamma_z=\{(s,t)\in \{0,1,\ldots,m\}\times \left( \bigcup_{k=0}^m J_k \right) \mid z \in U_{s,t}\}$. Then
\begin{equation}
    \begin{split}
        u g u^*= \text{diag}\left(\sum_{j\in J_0} h_{0,j}\,v_{x_{\ell(0,j)}}^{(0)}\, g_0\, \left(v_{x_{\ell(0,j)}}^{(0)}\right)^*, \ldots, \sum_{j\in J_m} h_{m,j}\,v_{x_{\ell(m,j)}}^{(m)}\, g_m\, \left(v_{x_{\ell(m,j)}}^{(m)}\right)^* \right).
    \end{split}
\end{equation}

Furthermore, 
\begin{equation}\label{eqn_cuntzeq}
       \begin{split}
          &\left\| \Bigg\{ \text{diag} \left(\sum_{j\in J_0} h_{0,j} f, \sum_{j\in J_1} h_{1,j} f,\ldots, \sum_{j\in J_m} h_{m,j} f \right)- u g u^* \Bigg\} (z) \right\|\\
          &= \left\| \Bigg\{ \text{diag} \left( \sum_{j\in J_0} h_{0,j} \left( f- v_{x_{\ell(0,j)}}^{(0)}\, g_0\, (v_{x_{\ell(0,j)}}^{(0)})^* \right),\ldots, \sum_{j\in J_m} h_{m,j} \left(f-v_{x_{\ell(m,j)}}^{(m)}\, g_m\, (v_{x_{\ell(m,j)}}^{(m)})^* \right) \right)\Bigg\}(z)\right\|\\
          &\leq \sum_{k=0}^m \sum_{j\in J_k} h_{k,j}(z) \left\|\left(f(z)- v_{x_{\ell(k,j)}}^{(k)}\, g_k(z)\, (v_{x_{\ell(k,j)}}^{(k)})^* \right)\right\| \\
          & < \sum_{(k,j)\in\Gamma_z} h_{k,j}(z)\frac{\varepsilon}{m+1}\\
          &\leq \frac{\varepsilon}{m+1}.
       \end{split}
   \end{equation}
   
   Thus, using Lemma \ref{L:CzBasic}~(\ref{L:CuntzDirectSum}) at the second step, and (\ref{eqn_cuntzeq}) and Lemma \ref{L:CzBasic}~(\ref{L:Cuntzcutcutdown}) at the fourth step, we see that,
 \begin{equation*}
       \begin{split}
         (f-\varepsilon)_+ &= \left(\sum_{j\in J_0} h_{0,j} f+ \sum_{j\in J_1} h_{1,j} f+\ldots+ \sum_{j\in J_m}h_{m,j} f -\varepsilon\right)_+\\
         &\precsim  \text{diag}\left(\left(\sum_{j\in J_0} h_{0,j}f-\frac{\varepsilon}{m+1}\right)_+, \left(\sum_{j\in J_1} h_{1,j}f-\frac{\varepsilon}{m+1}\right)_+,\ldots, \left(\sum_{j\in J_m}h_{m,j}f-\frac{\varepsilon}{m+1}\right)_+\right)\\
         &= \left(\text{diag}\left(\sum_{j\in J_0} h_{0,j} f, \sum_{j\in J_2} h_{1,j} f,\ldots, \sum_{j\in J_m}h_{m,j} f\right)-\frac{\varepsilon}{m+1}\right)_+\\
         &\precsim u g u^*\\
         &\precsim g.
         \end{split}
   \end{equation*}

It follows that for every $\varepsilon > 0$, we have 
$(f - \varepsilon)_+ \precsim g.$ 
Thus, we obtain
\[
f \precsim g_0 \oplus g_1 \oplus \ldots \oplus g_m,
\]
as required.

\end{proof}

\begin{prop}\label{comp_c(x,A)}
    Let $X$ be a compact metric space of covering dimension $m$, and let $A$ be a \ca{} with strict comparison. Then $C(X, A)$ has $m$-comparison.
\end{prop}
\begin{proof}
We may assume without loss of generality that $A$ is stable. Let $f,g_0,g_1,\ldots,g_m\in C(X,A)_+$. Assume that $[f]\leq_s [g_k]$ in $\Cu (C(X,A))$ for each $k=0,1,\ldots,m$. We need to show that \[
[f] \leq \sum_{k=0}^m [g_k]\quad\text{in}\quad\Cu(C(X,A)).
\]

To this end, see that for any $x\in X$, the evaluation map ${\rm ev}_x\colon C(X,A)\to A$ given by $f\mapsto f(x)$ is a $^*$-homomorphism. Thus, it preserves Cuntz subequivalence (and, in particular, the relation $\leq_s$ in the Cuntz semigroup). We have
\[
    [f(x)]\leq_s [g_k(x)]
\]
in $\Cu (A)$ for each $x\in X$ and $k=0,1,\ldots,m$.
Using $0$-comparison in $\Cu (A)$, we get $f(x)\precsim g_k(x)$ for each $x\in X$ and $k=0,1,\ldots,m$. Applying Lemma~\ref{lem:CuntzSubC(X,A)}, it follows that
\[
f \precsim g_0 \oplus g_1 \oplus \cdots \oplus g_m.
\] 
Consequently, in the Cuntz semigroup $\Cu(C(X,A))$,
\[
[f] \leq \sum_{k=0}^m [g_k],
\] 
as desired.
\end{proof}

\section{\texorpdfstring{Divisibility in $\Cu(C(X,A))$}{Divisibility in Cu(C(X,A))}}\label{sec:DivCuA}

Having established $m$-comparison for $C(X, A)$ in Proposition \ref{comp_c(x,A)}, we now turn our attention to divisibility. As outlined in the introduction, our strategy to prove pureness (and, consequently, strict comparison) of $C(X, A)$ relies on establishing a suitable divisibility property for its Cuntz semigroup that will give us access to dimension reduction arguments. In Corollary \ref{robert_analogy_2}, we show that if $X$ is finite-dimensional and $A$ is almost divisible, then $\Cu(C(X, A))$ inherits a specific form of divisibility that is sufficient for our purposes (namely, that of Theorem \ref{prp:Gen_DimRed}).

The following definition is from \cite[Paragraph~2.3]{RobTik17NucDimNonSimple}, which slightly differs from the original \cite[Definition~2.5]{Win12NuclDimZstable} but behaves better for non-simple \ca{s}. 

\begin{defn}\label{def_almost_divisibility}
    Let $A$ be a \ca{} and $n\in\N$. We say that $A$ (and its Cuntz semigroup $\Cu (A)$) is \emph{$n$-almost divisible} if for any pair $x',x\in\mathrm{Cu}(A)$ such that $x'\ll x$ and any $N\in \N$, there exists $y\in \mathrm{Cu}(A)$ such that 
    \[
     N y\leq x\quad{\mbox{and}}\quad x'\leq (N+1)(n+1)y.
    \]

    $A$ is said to be \emph{almost divisible} if it is $0$-almost divisible.
\end{defn}

\begin{prop}\label{robert_analogy}
 Let $X$ be a compact metric space of covering dimension $m \in\N$, and let $A$ be an almost divisible \ca{}. Then, for any $[f] \in \Cu(C(X,A))$, $\varepsilon > 0$, and $N \in \mathbb{N}$, there exists $[a]\in \mathrm{Cu}(C(X,A))$ such that 
\[
[(f - \varepsilon)_+]\leq (N+1)[a]
\quad\text{and}\quad 
N[a]\leq (m+1)^2 [f].
\]
\end{prop}
\begin{proof}
We may assume without loss of generality that $A$ is stable and, consequently, that $f\in C(X,A)_+$. Fix $N\in\N$. By compactness of $X$, we can find a finite open cover $\mathcal{V}$ of $X$ such that if $V \in \mathcal{V}$ and $x, y \in V$, then 

    \begin{equation}\label{norm_diff}
  \|f(x)-f(y)\|< \frac{\varepsilon}{8(m+1)}.
    \end{equation}
    Using that the covering dimension of $X$ is $m$, we can find a finite open refinement $\mathcal{U}$ of
order $m + 1$. Thus,  we can decompose it in the following way: $\mathcal{U}=\{{U_{0,j}}\}_{j\in J_0}\cup \{{U_{1,j}}\}_{j\in J_1} \cup \ldots \cup \{{U_{m,j}}\}_{j\in J_m}$ such that 
\[
U_{k,j}\cap U_{k,j'}=\emptyset, \quad\mbox{for}\,\, j\neq j'\,\,\text{and}\,\, k=0,1,\ldots,m. 
\]

Fix points $x_{k,j}\in U_{k,j}$ and for $k=0,1,\ldots,m$ define maps $\psi_k: C(X,A)\to \bigoplus_{j\in J_k} A$  as
\[
\psi_k(g)=\left(g(x_{k,j})\right)_{j\in J_k}.
\]
Choose a partition of unity subordinate to the cover $\mathcal{U}$, say $h_{k,j}:X\rightarrow [0,1]$, such that 
\[
\text{supp}(h_{k,j})\subseteq U_{k,j},\quad{\mbox{and}}\quad\sum_{k=0}^m\sum_{j\in J_{k}} h_{k,j}=1,
\]
and define maps $\varphi_k: \bigoplus_{j\in J_k} A\to C(X, A)$ as 
\[
\varphi_k\left(\left(a_{k,j}\right)_{j\in J_k}\right)= \sum_{j\in J_k} a_{k,j} h_{k,j}.
\]
Then, for $k=0,1,\ldots,m$, $\psi_k$ and $\varphi_{k}$ are cpc order zero maps such that
\begin{equation}\label{eqn_norm}
 \left\|f- \sum_{k=0}^{m} (\varphi_k\circ \psi_k)(f)\right\| \leq \frac{\varepsilon}{8(m+1)}<\frac{\varepsilon}{2},
\end{equation}
and by Part~(\ref{L:Cuntzcutcutdown}) of Lemma~\ref{L:CzBasic}, 
\begin{equation}\label{Phillips_large}
(f-\varepsilon)_+= \left(f-\frac{\varepsilon}{2}-\frac{\varepsilon}{2}\right)_+\precsim \left(\sum_{k=0}^m(\varphi_k\circ\psi_k)(f)-\frac{\varepsilon}{2}\right)_+. 
\end{equation}
    Since $\mathrm{Cu}(A)$ is almost divisible, by Definition~\ref{def_almost_divisibility} there exist $a_0, a_1\ldots,a_m$, where $a_k\in \left(\bigoplus_{j\in J_k} A\right)_+$, such that for $k=0,1,\ldots,m$,
     \[
N [a_k]\leq \left[\left(\psi_k(f)-\frac{\varepsilon}{8(m+1)}\right)_+\right]\quad\mbox{and}\quad\left[\left(\psi_k(f)-\frac{\varepsilon}{4(m+1)}\right)_+\right]\leq (N+1)[a_k].
     \]
Hence, for $k=0,1,\ldots,m$,
\begin{equation}\label{cuntz_inequality1}
\begin{split}
N [\varphi_k(a_k)]\leq & \left[\varphi_k\left(\left(\psi_k(f)-\frac{\varepsilon}{8(m+1)}\right)_+\right)\right],
\end{split}
\end{equation}
and
\begin{equation}\label{cuntz_inequality2}
    \begin{split}
\left[\varphi_k\left(\left(\psi_k(f)-\frac{\varepsilon}{4(m+1)}\right)_+\right)\right]&\leq  (N+1)\,[\varphi_k(a_k)].
\end{split}
      \end{equation}
Furthermore, since for a fixed $k$, $\{U_{k,j}\}_{j\in J_k}$ are disjoint, it follows from (\ref{norm_diff}) that for each $x\in X$
      \[
\Bigg\{\sum_{j\in J_k} \left(f(x_{k,j})-\frac{\varepsilon}{8(m+1)}\right)_+ h_{k,j}\Bigg\}(x) \precsim f(x).
      \]
Using Lemma~\ref{lem:CuntzSubC(X,A)} at the second step, we see that 
  \[
 \varphi_k\left(\left(\psi_k(f)-\frac{\varepsilon}{8(m+1)}\right)_+\right)=\sum_{j\in J_k} \left(f(x_{k,j})-\frac{\varepsilon}{8(m+1)}\right)_+ h_{k,j} \precsim  f\otimes 1_{m+1},
  \]
     and by (\ref{cuntz_inequality1}), we get

\begin{equation}\label{homomorphism}
    \begin{split}
        N [\varphi_k(a_k)]\leq & \left[\varphi_k\left(\left(\psi_k(f)-\frac{\epsilon}{8(m+1)}\right)_+\right)\right] \leq (m+1)[f].
    \end{split}
\end{equation}
  Furthermore,  for each $x\in X$ see that
  \[
 \left((\varphi_k\circ \psi_k)(f)-\frac{\varepsilon}{4(m+1)}\right)_+ (x) \leq \varphi_k\left(\left(\psi_k(f)-\frac{\varepsilon}{4(m+1)}\right)_+\right) (x),
  \]
  hence
  \[
   \left((\varphi_k\circ \psi_k)(f)-\frac{\varepsilon}{4(m+1)}\right)_+  \leq \varphi_k\left(\left(\psi_k(f)-\frac{\varepsilon}{4(m+1)}\right)_+\right) \quad\text{in}\quad C(X,A).
  \]
 Thus, we further see that 
 \[
 \left[\left((\varphi_k\circ \psi_k)(f)-\frac{\varepsilon}{4(m+1)}\right)_+ \right]\leq \left[\varphi_k\left(\left(\psi_k(f)-\frac{\varepsilon}{4(m+1)}\right)_+\right)\right],
 \]
 and so by (\ref{cuntz_inequality2}), it follows that

 \begin{equation}\label{ineq_cpc}
     \begin{split}
       \left[\left((\varphi_k \circ\psi_k)(f)-\frac{\varepsilon}{4(m+1)}\right)_+ \right]\leq  (N+1)\,[\varphi_k(a_k)].
     \end{split}
 \end{equation}
 
   Finally, define $a:=\text{diag}\left( \varphi_0(a_0), \varphi_1(a_1), \ldots, \varphi_m(a_m)\right) \in M_{m+1}(C(X,A))$. Then, by (\ref{homomorphism}),
     \begin{equation}\label{first_ine}
         N[a]= \sum_{k=0}^m N [
         \varphi_k(a_k)]\leq (m+1)^2[f].
     \end{equation}
     Lastly, by using (\ref{Phillips_large}) at the first step, Part~(\ref{L:CuntzDirectSum}) of Lemma~\ref{L:CzBasic} at the second step and (\ref{ineq_cpc}) at the fifth step, we see that
     \begin{equation}\label{second_ine}
         \begin{split}
             [(f-\varepsilon)_+]\leq \left[\left(\sum_{k=0}^m(\varphi_k\circ\psi_k)(f)-\frac{\varepsilon}{2}\right)_+\right]
        &\leq\left[\sum_{k=0}^m\left((\varphi_k\circ\psi_k)(f)-\frac{\varepsilon}{2(m+1)}\right)_+\right]\\
             & \leq \sum_{k=0}^m \left[ \left((\varphi_k\circ\psi_k)(f)-\frac{\varepsilon}{2(m+1)}\right)_+\right]\\
             & \leq \sum_{k=0}^m \left[ \left((\varphi_k\circ\psi_k)(f)-\frac{\varepsilon}{4(m+1)}\right)_+\right]\\
             &\leq (N+1) \sum_{k=0}^m[\varphi_k(a_k)]\\
             &= (N+1)[a].
         \end{split}
     \end{equation}

     Hence, from (\ref{first_ine}) and (\ref{second_ine}) we get the desired conclusion.
\end{proof}

\begin{cor}\label{robert_analogy_2}
    Let $X$ be a finite-dimensional compact metric space, and let $A$ be an almost divisible \ca{}. Then, there exists $M\in\N$ such that, for every $N\geq 1$ and every pair $x'\ll x$ in $\Cu (C(X,A))$, there exists $y\in\Cu (C(X,A))$ such that
    \[
        x'\ll Ny\ll Mx.
    \]
\end{cor}
\begin{proof}
    Set $M=2(\dim (X)+1)^2$ and let $N\in\N$. If $N=1$, any $y\in\Cu (A)$ such that $x'\ll y\ll x$ satisfies the required condition. Thus, we may assume $N>1$.
    As in the previous proof, we may also assume that $A$ is stable. Find $w\in \Cu (C(X,A))$ such that $x'\ll w\ll x$, and let $w=[f]$ with $f\in C(X,A)_+$. By the definition of compact containment, there exists $\varepsilon>0$ such that $x'\ll [(f-\varepsilon)_+]$.  Now apply Proposition \ref{robert_analogy} to $f$, $\varepsilon$ and $N-1$ to find $y\in \Cu (C(X,A))$ such that 
    \[
        [(f-\varepsilon)_+]\leq Ny\quad\text{and}\quad (N-1)y\leq (\dim(X)+1)^2[f].
    \]

    We have
    \[
        x'\ll [(f-\varepsilon)_+]\leq Ny\leq 2(N-1)y\leq Mw\ll Mx,
    \]
    as desired.
\end{proof}

\section{Dimension reduction}

In this section, we combine the results of the previous sections to establish Theorem \ref{prp:MainThm1}. We first prove Theorem \ref{prp:Gen_DimRed}, an abstract dimension reduction result showing that any \ca{} satisfying $m$-comparison and the divisibility established in Section \ref{sec:DivCuA}, together with the Global Glimm Property (Paragraph \ref{pgr:GGP}), is necessarily pure. We then apply this general framework to the context of $C(X, A)$. In particular, we obtain pureness when: $A$ is simple and pure (Corollary~\ref{Purehomogenous}); or when $A$ is pure and satisfies a residual finiteness condition ---namely, that the Cuntz semigroup of any quotient of $A$ contains no nonzero compact properly infinite elements (Theorem~\ref{prp:MainNonSimp}).

Following \cite[Definition~2.6]{Win12NuclDimZstable}, pure \ca{s} are those that enjoy good comparison and divisibility properties:

\begin{defn}\label{dfn:pure}
    A \ca{} $A$ (and its Cuntz semigroup $\Cu (A)$) is said to be \emph{pure} if $\Cu (A)$ is almost unperforated and almost divisible. 
\end{defn}

To prove Theorem~\ref{prp:MainNonSimp}, we will make use of the following properties, which are satisfied in the Cuntz semigroup of every \ca{} (see \cite[Proposition~5.1.1]{robert2013cone} and \cite[Proposition~2.2]{antoine2021edwards} respectively):

\begin{enumerate}[label=(\thecuax), ref=\thecuax, leftmargin=*]
\setcounter{cuax}{5}
 \stepcounter{cuax}\item  \label{O6}
Given $x'\ll x\leq y+z$, there exist $v$ and $w$ such that $v\ll  x,y$, and $w\ll x,z$ satisfying $x'\ll v +w$.
 \stepcounter{cuax}\item \label{O7}
Given $x_1'\ll x_1\leq w$ and $x_2'\ll x_2\leq w$, there exists $x\in S$ such that
\[
x_1',x_2'\ll x \ll w,x_1+x_2.
\]
\end{enumerate}

The next set of results follow from induction on (\ref{O6}) and (\ref{O7}). Although these are well known among experts, we were not able to locate a proof in the literature. We provide it here for the convenience of the reader:

\begin{lem}\label{prp:RefO6}
    Let $A$ be a \ca{} and let $x', x,y_1,\ldots ,y_n\in \Cu (A)$ satisfy
    \[
    x'\ll x \leq y_1+\ldots +y_n.
    \]
    
    Then, there exists $z_1,\ldots,z_n\in \Cu (A)$ such that,  for each $i=1,2,\ldots n$, $z_i\ll x,y_i$ and  
    \[
    x' \ll z_1 +z_2 + \ldots z_n.
    \]
\end{lem}
\begin{proof}
    The case $n=2$ follow from ~(\ref{O6}). Thus, assume the statement holds for some $n\geq 2$, and suppose 
    \[
    x'\ll x \leq y_1+\ldots +y_{n+1}.
    \]

    Use ~(\ref{O6}) for $y=y_1$ and $z=y_2+\ldots +y_{n+1}$ to find $z_1\ll x,y_1$ and $w\ll x,y_2+\ldots +y_{n+1}$ such that $x'\ll z_1+w$. Take $w'\in\Cu (A)$ such that $w'\ll w$ and $x'\ll z_1+w'$, and apply the induction hypothesis to $w'\ll w\leq y_2+\ldots +y_{n+1}$ to obtain elements $z_2,\ldots ,z_{n+1}$ such that $z_i\ll w,y_i$ and $w'\ll z_2+\ldots +z_{n+1}$.

    Since $w\ll x$, it follows that $z_1,\ldots ,z_{n+1}$ satisfy the desired conditions.
\end{proof}

\begin{lem}
\label{prp:RefO7}
Let $A$ be a \ca{} and let $x_j',x_j,w\in \Cu (A)$ satisfy
\[
x_j'\ll x_j\leq w
\]
for $j=1,\ldots,n$.
Then, there exists $x\in \Cu (A)$ such that
\[
x_1',\ldots,x_n'\ll x \ll w, x_1+\ldots+x_n.
\]
\end{lem}
\begin{proof}
    The case $n=2$ follow from ~(\ref{O7}). As before, assume the statement holds for some $n\geq 2$, and let $x_j',x_j,w\in \Cu (A)$ satisfy
    \[
    x_j'\ll x_j\leq w
    \]
    for $j=1,\ldots,n+1$.

    By the induction hypothesis applied to the first $n$ inequalities, we find $r\in \Cu (A)$ such that $x_i'\ll r\ll w,x_1+\ldots +x_n$ for $i=1,\ldots ,n$. Take $r'\in\Cu (A)$ such that $r'\ll r$ such that $x_i'\ll r'$ for all $i=1,\ldots ,n$, and apply ~(\ref{O7}) to the pair $r'\ll r\leq w$ and $x_{n+1}'\ll x_{n+1}\leq w$ to find $x\in\Cu (A)$ such that $r',x_{n+1}'\ll x\ll w,r+x_{n+1}$.

    Using that $r\leq x_1+\ldots +x_n$ and $x_i'\leq r$ for each $i\leq n$, we see that $x$ satisfies the required conditions.
\end{proof}

\begin{lem}\label{prp:RedandAmp}
    Let $A$ be a $\mathrm{C}^*$-algebra and let $M\in\mathbb{N}$. Assume that, for every $N\geq 1$ and every pair $x',x\in \Cu (A)$ such that $x'\ll x$, there exists $y\in\mathrm{Cu} (A)$ satisfying
    \[
        x'\ll Ny\ll Mx.
    \]

    Then, there exists $M_1\in\N$ such that, for every $N\geq 1$ and every pair $x',x\in \Cu (A)$ such that $x'\ll x$, there exists $z\in\mathrm{Cu} (A)$ with
    \[
        z\ll x\quad\text{and}\quad 
        x'\ll Nz\ll M_1x.
    \]

    In particular, if $x=[a]$, $z$ can be taken to be of the form $z=[c]$ with $c\in \overline{a(A\otimes\mathcal{K})a}$.
\end{lem}
\begin{proof}
    Set $M_1=2M^3$. Let $x',x\in \Cu (A)$ be such that $x'\ll x$ and let $N\geq 1$. Note that, if $N\leq M$, we can set $z=x''$ for some $x''\in\Cu (A)$ such that $x'\ll x''\ll x$. Thus, we may assume that $N>M$ and find $k\geq 1$ such that $kM\leq N\leq (k+1)M$.

    Applying the assumption in the lemma for $x'\ll x$ and $k$, find $y\in\Cu (A)$ such that $x'\ll ky\ll Mx$. Let $y'\in\Cu (A)$ be such that $y'\ll y$ and $x'\ll ky'$.
     
    Using Lemma \ref{prp:RefO6} for $y'\ll y\leq Mx$, find $y_1,\ldots ,y_{M}$ in $\mathrm{Cu} (A)$ such that
    \[
        y'\leq y_1+\ldots +y_{M}\quad\text{and}\quad 
        y_j\leq y,x
    \]
    for each $j$.

    In particular, note that $x'\ll ky'\leq ky_1+\ldots +ky_{M}$. Now for each $j$ choose $y_j'$ such that 
    \[
        y_j'\ll y_j\quad \text{and}\quad x'\leq ky_1'+\ldots +ky_{M}'.
    \]
    Note that each $y_j'$ satisfies $y_j'\ll y_j\leq x$. Using Lemma \ref{prp:RefO7}, there exists $z\in\mathrm{Cu} (A)$ such that
    \[
        y_j'\ll z\ll x, y_1+\ldots +y_{M}.
    \]

    Then,
    \[
        x'\leq ky_1'+\ldots +ky_{M}'\ll (kM)z\leq N z 
    \]
    and
    \[
        Nz\leq 2(kM)z
        \ll 2(kM)(y_1+\ldots +y_{M})\leq 2(kM)(My)=2M^2(ky)\leq 2M^3 x=M_1x,
    \]
    as desired.

    Now, if $x=[a]$, note that $z\leq [a]$ and that $x'\ll Nz$. Upon passing to a small enough cut-down of $z$, we may assume by Part~(\ref{L:CutDowninHer}) of Lemma~\ref{L:CzBasic} that $z=[c]$ with $c\in \overline{(a(A \otimes \mathcal{K}) a)_+}$.
\end{proof}

\begin{pgr}[The Global Glimm Property]\label{pgr:GGP}
    A \ca{} $A$ is said to satisfy the \emph{Global Glimm Property} if for every $a\in A_+$ and $\varepsilon>0$ there exists a square-zero element $r\in\overline{aAa}$ such that $(a-\varepsilon)_+\in \overline{ArA}$. This property was originally introduced by Kirchberg and R\o{}rdam in their study of non-simple purely infinite \ca{s} \cite[Definition~4.12]{KirRor02InfNonSimpleCalgAbsOInfty}, but has recently found various and deep applications in the context of dimension reduction phenomena; see, for example, \cite{ThiVil23Glimm,Vil25:IntroGGP}.

    As shown in \cite[Theorem~3.6]{ThiVil23Glimm}, the property can be characterized in terms of a divisibility condition on the Cuntz semigroup: $A$ has the Global Glimm Property if and only if $\Cu (A)$ is $(2,\omega )$-divisible, that is, for  every pair $x',x\in\Cu (A)$ such that $x'\ll x$, there exists $y\in\Cu (A)$ and $n\in\N$ such that $2y\leq x$ and $x'\leq ny$.
\end{pgr}

\begin{theorem}\label{prp:Gen_DimRed}
    Let $A$ be a \ca{}, and let $m,M\in\mathbb{N}$. Assume that $A$ has $m$-comparison and that, for every $N\geq 1$ and every pair of elements $x', x\in \Cu (A)$ such that $x'\ll x$, there exists $y\in\mathrm{Cu} (A)$ satisfying 
    \[
        x'\ll Ny\ll M x.
    \]

    Then, the following are equivalent:
    \begin{itemize}
        \item[(i)] $A$ has the Global Glimm Property;
        \item[(ii)] there exists $L\in\N$ such that, for every pair $x',x\in\Cu (A)$ with $x'\ll x$, there exist $y_0,y_1\in\mathrm{Cu} (A)$ satisfying
        \[
        y_0+y_1\leq x,\quad\text{and}\quad 
        x'\ll Ly_0, Ly_1;
        \]
        \item[(iii)] $A$ is pure.
    \end{itemize}
\end{theorem}
\begin{proof}
    Using Lemma \ref{prp:RedandAmp}, there exists $M_1\in\N$ such that, for every $N\geq 1$ and every pair of elements $x'\ll x$ in $\Cu (A)$, there exists $z\in\mathrm{Cu} (A)$ such that 
    \[
        z\ll x\quad\text{and}\quad 
        x'\ll Nz\ll M_1 x.
    \]

    \noindent\textbf{(i)~implies~(ii).} The proof of this implication follows the same steps as \cite[Lemma~6.2]{AntPerThiVil24arX:PureCAlgs} by carefully applying our divisibility results instead of the ones obtained in \cite{AntPerThiVil24arX:PureCAlgs}. We encourage the reader to follow the proof of \cite[Lemma~6.2]{AntPerThiVil24arX:PureCAlgs} and, in parallel, apply the modifications listed below. Specifically, the necessary adjustments are:

    First, the constant $L$ is set to be $2(m+1)(M_1+2)$. Further, one notes that both the Global Glimm Property and the divisibility and comparison properties satisfied by $A$ are separably inheritable. Using \cite[Proposition~6.1]{ThiVil21DimCu2}, there exists a separable sub-\ca{} $B$ of $A$ such that: $B$ satisfies the Global Glimm Property and the same comparison and divisibility conditions of $A$; the induced inclusion $\Cu(\iota )\colon\Cu (B)\to\Cu (A)$ is an order-embedding; and $x',x\in {\rm im}(\Cu (\iota))$. Thus, as in \cite[Lemma~6.2]{AntPerThiVil24arX:PureCAlgs}, it suffices to prove (i)~implies~(ii) for separable \ca{s} and we may assume $A$ to be separable.

    When defining $f\in\mathrm{Cu} (A)$, we apply Lemma \ref{prp:RedandAmp} with $N=M_1+2$. This will give
    \begin{equation}\label{Eqn:ana6.2}
        f\ll e',\quad\text{and}\quad 
        e''\ll 
        (M_1+2)f\ll 
        M_1e'.
    \end{equation}

    Afterwards, one continues the proof until the ninth displayed equation, where with our new constants we obtain the inequality
    \[
        x'\ll 2e''\leq 
        2(M_1+2)f'=2(M_1+2)y_0\leq 
        Ly_0.
    \]

    Finally, multiply equation (2) in \cite[Lemma~6.2]{AntPerThiVil24arX:PureCAlgs} in the proof by $M_1+2$, and using (\ref{Eqn:ana6.2}) at the second step,  we get
    \[
        (M_1+2)e'\leq (M_1+2)f+(M_1+2)g\leq M_1e'+(M_1+2)g
    \]
    instead of equation (3) in \cite[Lemma~6.2]{AntPerThiVil24arX:PureCAlgs}.

    Continuing the proof of \cite[Lemma~6.2]{AntPerThiVil24arX:PureCAlgs}, one can show ---in the notation of Remark \ref{rem:StCompAlmUnperf}--- that there exists $\varepsilon\in (0,1)$ such that $d_\tau(e')\leq (1-\varepsilon)(M_1+2)d_\tau(y_1)$ for any $[0,\infty]$-valued lower-semicontinuous $2$-quasitrace $\tau$ (this $\varepsilon$ does not appear in the proof of \cite[Lemma~6.2]{AntPerThiVil24arX:PureCAlgs}, but follows from the fact that $d_\tau (e')\leq 2(1-\varepsilon)d_\tau (e')\leq (1-\varepsilon)(M_1+2)d_\tau (y_1)$ whenever $d_\tau (e')$ is finite).

    Using $m$-comparison and \cite[Lemma~1]{robert2010nuclear}, we get \[e'\leq (m+1)(M_1+2)y_1.\] It then follows that $x'\ll 2e'\leq 2(m+1)(M_1+2)y_1=Ly_1$, as desired.\vspace{0.25cm}

    \noindent\textbf{(ii) implies (iii).} Assuming (ii), an induction argument proves the following Claim (see the proof of \cite[Lemma~6.3]{AntPerThiVil24arX:PureCAlgs}):\vspace{0.25cm}

    \noindent\textbf{Claim.} For every $l\in\N$ there exists $L_1\in\N$ such that, for every pair $x'\ll x$, there exist $y_0,\ldots ,y_l\in\mathrm{Cu} (A)$ such that
    \[
        y_0+\ldots +y_l\leq x,\quad\text{and}\quad 
        x'\ll L_1y_j
    \]
    for each $j$.\vspace{0.25cm}

    Following the strategy of \cite[Proposition~6.4]{AntPerThiVil24arX:PureCAlgs}, we will prove that $A$ is $n$-almost divisible for $n=2L_1M_1$, where $L_1$ is the number given by the Claim for $l=m$. The result will then be a consequence of \cite[Theorem~5.7]{AntPerThiVil24arX:PureCAlgs}.

   Fix $k\geq 1$, and take a pair $x'\ll x$ in $\mathrm{Cu} (A)$. Apply Lemma \ref{prp:RedandAmp} for $N=2kL_1M_1$. We obtain $z\in\mathrm{Cu} (A)$ such that
    \[
        z\ll x,\quad\text{and}\quad 
        x'\ll (2kL_1M_1)z\ll M_1x.
    \]

    Choose $z'\in\Cu (A)$ such that $z'\ll z$ and $x'\ll (2kL_1M_1)z'$. Then, choose $x''\in\Cu (A)$ such that $x''\ll x$ and $(2kL_1M_1)z'\ll M_1x''$.

    Apply the Claim to $x''\ll x$ and $l=m$ to obtain $y_0,\ldots,y_m\in\Cu (A)$ such that
    \[
        y_0+\ldots +y_m\leq x,\quad\text{and}\quad 
        x''\ll L_1y_j
    \]
    for each $j$.

    For each fixed $j$, we have
    \[
        (L_1M_1+1)(kz')\leq (2kL_1M_1)z'\ll M_1x''\ll (L_1M_1)y_j,
    \]
    which implies $kz'<_s y_j$ for every $j$ (Notation \ref{ntn:s_below}). Thus, using $m$-comparison we have
    \[
        kz'\leq y_0+\ldots +y_m\leq x
    \]
    and
    \[
        x'\ll (2kL_1M_1)z'\leq (k+1)(2L_1M_1+1)z'
        = (k+1)(n+1)z'.
    \]

    This shows that $A$ is $n$-almost divisible, as desired.\vspace{0.25cm}

    \noindent\textbf{(iii) implies (i).} Note that almost divisibility implies $(2,\omega )$-divisibility. Thus, any pure \ca{} has the Global Glimm Property by the comments in Paragraph \ref{pgr:GGP}.
\end{proof}

\begin{rem}
    It is plausible that any \ca{} $A$ satisfying $m$-comparison and the divisibility condition of Theorem \ref{prp:Gen_DimRed} automatically possesses the Global Glimm Property (and is therefore pure). Indeed, any such algebra is \emph{nowhere scattered}—meaning no nonzero ideal of a quotient is elementary. Whether every nowhere scattered \ca{} enjoys the Global Glimm Property is a major open question, known as the \emph{Global Glimm Problem}; see the introduction of \cite{ThiVil23Glimm} for an overview to the problem.
\end{rem}

\begin{theorem}\label{prp:CXAPureIfGGP}
    Let $X$ be a finite-dimensional compact metric space and let $A$ be a pure \ca{}. Then, $C(X,A)$ is pure if and only if $C(X,A)$ has the Global Glimm Property.
\end{theorem}
\begin{proof}
    Proposition \ref{comp_c(x,A)} shows that $C(X,A)$ has $m$-comparison for $m=\dim (X)$. Further, it follows from Corollary \ref{robert_analogy_2} that $C(X,A)$ satisfies the divisibility condition in Theorem \ref{prp:Gen_DimRed}. Thus, Theorem \ref{prp:Gen_DimRed} gives the desired equivalence.
\end{proof}

\begin{cor}\label{Purehomogenous}
    Let $A$ be a simple pure \ca{}, and let $X$ be a compact metric space. Then, $C(X,A)$ is pure.
\end{cor}
\begin{proof}
    If $X$ is finite-dimensional, $C(X,A)$ has the Global Glimm Property by \cite[Theorem~4.3]{BlaKir04GlimmHalving}. Thus, the result follows from Theorem \ref{prp:CXAPureIfGGP}.

    If $X$ is arbitrary, write $C(X,A)$ as $\lim_\lambda C(X_\lambda ,A)$ with $X_\lambda$ finite-dimensional. Since pureness passes to inductive limits \cite[Theorem~3.8]{perera2025extensions}, this drops the finite-dimensionality assumption from $X$.
\end{proof}

We now move our attention to non-simple \ca{s}. Here, instead of focusing on the Global Glimm Property, we study when condition (ii) in Theorem \ref{prp:Gen_DimRed} is satisfied. Building on the ideas from \cite[Proposition~3.4]{RobTik17NucDimNonSimple}, we will see that $C(X,A)$ satisfies (ii) whenever $A$ is pure and has no quotient whose Cuntz semigroup contains a nonzero, compact, properly infinite element. Examples of \ca{s} satisfying this `residual finiteness' condition include residually stably finite \ca{s} (e.g., stable rank one \ca{s}), since each of their quotients have a stably finite Cuntz semigroup. For \ca{s} of finite nuclear dimension, the condition is equivalent to having no simple purely infinite quotients \cite[Lemma~3.3]{RobTik17NucDimNonSimple}.

\begin{prop}\label{prp:CXNoQuot}
    Let $X$ be a finite-dimensional compact metric space, and let $A$ be a pure \ca{}. Assume that $A$ has no quotient whose Cuntz semigroup contains a nonzero, compact, properly infinite Cuntz class. Then, there exists $L\in\N$ such that, for every pair $x',x\in\Cu (C(X,A))$ with $x'\ll x$, there exist $y_0,y_1\in\mathrm{Cu} (C(X,A))$ satisfying
        \[
        y_0+y_1\leq x,\quad\text{and}\quad 
        x'\ll Ly_0, Ly_1.
        \]
\end{prop}
\begin{proof}
    We may assume without loss of generality that $A$ is stable, and that there exists $\varepsilon>0$ such that that $x=[f]$ and $x'=[(f-\varepsilon)_+]$ for some $f\in C(X,A)_+$. We adapt the argument from \cite[Proposition~3.4]{RobTik17NucDimNonSimple}. However, distinct technical difficulties arise in the present context that demand additional care. For this reason, we present the full proof in detail:

    Let $m$ be the dimension of $X$, and let $M_1$ be the constant given by Lemma \ref{prp:RedandAmp} (and Corollary \ref{robert_analogy_2}). Set $L=\max\{2M_1,3(m+1)\}$.
    
    Using Lemma \ref{prp:RedandAmp} with $N=2M_1$, there exists $c\in\overline{fC(X,A)f}_+$ such that 
    \[
        [(f-\frac{\varepsilon}{2})_+]\ll 2M_1[c]\ll M_1[f]
    \]
    in $\Cu (C(X,A))$.

    Let $\delta >0$ be such that
    \[
    [(f-\varepsilon)_+]\leq 2M_1[(c-\delta)_+].
    \]
    
    Set
    \[
    d^0:=g_\delta (c)\quad\text{and}\quad d^1:=(1-g_{\delta /2}(c))^{1/2}f(1-g_{\delta/2}(c))^{1/2},
    \]
    where $g_\delta$ is the piecewise linear function that is $0$ from $0$ to $\delta /2$ and $1$ from $\delta$ onward. Further, set $y_0:=[d^0]$ and $y_1:=[d^1]$, where note that $y_0+y_1\leq [f]=x$.

    We get 
    \[
        x'=[(f-\varepsilon)_+]\leq 2M_1[d^0]\leq L[d^0]=Ly_0.
    \]
    Additionally, we also have 
    \begin{equation}\label{eq:510}
        [f]\leq
        [g_{\delta/2}(c)]+[d^1].
    \end{equation}
    
    Now let $\bar{\varepsilon}>0$ be such that
    \[
    2M_1[g_{\delta/2} (c)]\leq M_1[(f-\bar{\varepsilon})_+].
    \]
    Multiplying (\ref{eq:510}) by $2M_1$, we obtain 
    \begin{equation}\label{eq:5102}
        2M_1[f]\leq 2M_1[g_{\delta/2}(c)]+2M_1[d^1]\leq M_1[(f-\bar{\varepsilon})_+]+2M_1[d^1].
    \end{equation}

    Evaluating at any point $t\in X$ is a $^*$-homomorphism and thus preserves Cuntz subequivalence. In particular, one gets
    \[
        2M_1[f(t)]\leq M_1[(f-\bar{\varepsilon})_+(t)]+2M_1[d^1(t)]
    \]
    in $\Cu (A)$ for each $t\in X$.

    Now fix $t\in X$. Passing to the quotient $A/\langle d^1(t)\rangle$, we get 
    \[
        2M_1[\pi(f(t))]\leq M_1[\pi((f-\varepsilon)_+(t))].
    \]

    Thus, the element $M_1[\pi (f(t))]$ is a compact, properly infinite class in $\Cu (A/\langle d^1(t)\rangle)$. Thus, $[\pi (f(t))]=0$ by our assumptions. In other words, $f(t)$ belongs to the ideal generated by $d^1(t)$, which in terms of Cuntz classes is written as $[f(t)]\leq \infty [d^1(t)]:=\sup_n n[d^1 (t)]$ in $\Cu (A)$.

    Using once again that evaluating at $t$ is a $^*$-homomorphism, it follows that it preserves suprema of Cuntz classes. Thus, we have
    \[
        [f(t)]\leq \infty [d^1(t)]=\sup_n n[d^1(t)]=\Cu ({\rm ev}_t)\left(\sup_n n[d^1]\right)=\Cu ({\rm ev}_t)(\infty[d^1]).
    \]

    Let $D^1\in C(X,A)_+$ be such that $[D^1]=\infty [d^1]$. The previous inequality gives 
    \[
        f(t)\precsim D^1(t)
    \]
    in $A$ for each $t\in X$.

     Thus, it follows from Lemma \ref{lem:CuntzSubC(X,A)} that $[f]\leq (m+1) [D^1]$ in $C(X,A)$. In other words, $[f]\leq (m+1) (\infty [d^1])=\infty [d^1]$ in $\Cu (C(X,A))$. In particular, since \[[g_{\delta/2} (c)]\ll [f],\] a finite multiple of $[d^1]$ is Cuntz above $[g_{\delta /2} (c)]$ and, since \[[f]\leq [g_{\delta /2}(c)]+[d^1],\] a finite multiple of $[d^1]$ is above $[f]$. From this, it follows that $d_\tau ([d^1])=\infty$ whenever $d_\tau ([f])=\infty$ for any $[0,\infty]$-valued lower semi-continuous $2$-quasitrace $\tau$ on $C(X,A)$ (recall the definition from Remark \ref{rem:StCompAlmUnperf}).
     
     Now take any ($[0,\infty]$-valued lower semi-continuous) $2$-quasitrace $\tau$ on $C(X,A)$. Evaluating $d_\tau$ on (\ref{eq:5102}) we get 
     \[
        2d_\tau ([f])\leq d_\tau ([f])+ 2d_\tau ([d^1]).
     \]

     Thus, we deduce that $d_\tau ([f])\leq 2d_\tau ([d^1])$. Indeed, if $d_\tau ([f])$ is finite, we simply subtract the right hand side from the left hand side. Otherwise, if $d_\tau ([f])=\infty$, we know that $d_\tau ([d^1])=\infty$.

     Let $\gamma>0$ be such that $2\leq 3(1-\gamma)$. Using that $C(X,A)$ has $m$-comparison and that $d_\tau ([f])\leq (1-\gamma)d_\tau (3[d^1])$, we deduce from \cite[Lemma~1]{robert2010nuclear} that $x'\leq [f]\leq 3(m+1)[d^1]\leq L [d^1]=Ly_1$.
\end{proof}

\begin{theorem}\label{prp:MainNonSimp}
    Let $X$ be a compact metric space, and let $A$ be a pure \ca{}. Assume that $A$ has no quotient whose Cuntz semigroup contains a nonzero, compact, properly infinite Cuntz class. Then, $C(X,A)$ is pure.
\end{theorem}
\begin{proof}
    If $X$ is finite-dimensional, the result follows directly from Proposition \ref{prp:CXNoQuot} together with Theorem \ref{prp:Gen_DimRed} (together with Proposition \ref{comp_c(x,A)} and Corollary \ref{robert_analogy_2}).

    If $X$ is not finite-dimensional, proceed as in the proof of Corollary \ref{Purehomogenous} and write $C(X,A)$ as $\lim_\lambda C(X_\lambda ,A)$ with $X_\lambda$ finite-dimensional to reduce to the finite-dimensional case.
\end{proof}

\section{Pureness of Recursive Subhomogeneous Algebras}

In this section, we use the permanence properties established in \cite{perera2025extensions}
to show that pureness extends to a broader class of \ca{s}. In particular,
we prove that any recursive subhomogeneous $C^*$-algebra over a unital, simple,
pure $C^*$-algebra $D$ is itself pure, and that the same holds for the tensor
product $R \otimes D$, where $R$ is a recursive subhomogeneous algebra over
$\mathbb{C}$. As a consequence, we deduce that a locally trivial continuous
$C(X)$-algebra with a simple, unital, pure fibre is pure, and that the tensor
product of a unital separable ASH-algebra with a simple, pure \ca{} is also pure. Analogous results hold for non-simple, pure, residually stably finite \ca{s}; see Remark \ref{rem:Final}.

We begin by recalling the definition of recursive subhomogeneous 
algebras built over a \ca{}~$D$. For more details the reader is referred to Section 3 in  \cite{archey2020structure}.

We start with the following definition:
\begin{defn}\label{pullback}
Let $A$, $B$ and $C$ be \ca{s}, and let $\ph \colon
A \to C$ and $\psi \colon B \to C$ be $^*$--homomorphisms.
Then the associated {\emph{pullback \ca{}}} $A
\oplus_{C, \ph, \psi} B$ is defined by
\[
A \oplus_{C, \ph, \psi} B
 = \bset{ (a, b) \in A \oplus B \colon \ph (a) = \psi (b) }.
\]
\end{defn}

Note that $A \oplus_{C, \ph, \psi} B$ can be described by the commutative diagram 
\[
\begin{tikzcd}[row sep=3.5em, column sep=4.5em]
A \oplus_{C,\varphi,\psi} B
    \arrow{r}{\pi_B}
    \arrow{d}[swap]{\pi_A}
&
B \arrow{d}{\psi}
\\
A \arrow{r}{\varphi}
&
C
\end{tikzcd}
\]

\noindent
where \(\pi_A(a,b)=a\) and \(\pi_B(a,b)=b\).

The following lemma describes how pureness behaves with respect to pullbacks.

\begin{lem}\label{PurePullback}
Let $A$, $B$ and $C$ be \ca{s},
and let $\ph \colon A \to C$ and $\psi \colon B \to C$
be $*$-homomorphisms.
Let $P = A \oplus_{C, \ph, \ps} B$ be the associated pullback. 
If $\ps$
is surjective and both $A$ and $B$ are pure, then $P$ is
also pure.
\end{lem}
\begin{proof}
As shown in \cite[Theorem 4.11]{perera2025extensions}, an extension of \ca{s} 
\[
0\rightarrow J\rightarrow A\rightarrow B\rightarrow 0
\]
satisfies that $A$ is pure if and only if $J$ and $B$ are pure. 
 
We will make use of this permanence property to conclude the desired result. To this end, consider the following two extensions coming from the pullback construction: The first extension follows by surjectivity of $\ps$. Furthermore, it is straightforward to check that surjectivity of $\ps$ implies surjectivity of $\pi_A$, which gives the second extension. 
Thus, we have
\begin{equation*}
    \begin{split}
        &0\rightarrow \text{Ker}(\ps)\rightarrow B\xrightarrow{\ps} C\rightarrow 0,\,\,\mbox{and}\\
        &0\rightarrow \text{Ker}(\ps)\xrightarrow {i}P\xrightarrow {\pi_A} A\rightarrow 0,
    \end{split}
\end{equation*}
 where $i(b)=(0,b)$, and $\pi_A(a,b)=a$. 

Then, $\text{ker}(\ps)$ is pure since $B$ is pure, and $P$ is pure since $\text{Ker}(\ps)$ and $A$ are. 
\end{proof}

The following definition is from \cite{archey2020structure}, and generalizes the class of recursive subhomogeneous algebras by taking $D=\mathbb{C}$ (see \cite{PhRsha1}).

\begin{defn}[{\cite[Definition~3.2]{archey2020structure}}]\label{RSHA}
For a simple unital \ca{} $D$, the class $\mathcal{R}$ of 
{\emph{recursive subhomogeneous algebras over $D$}} is
the smallest class of \ca{s} closed under isomorphism such that:
\begin{enumerate}
\item\label{9223_RSHA_1}
$C (X, M_{n} (D)) \in \mathcal{R}$ whenever $X$ is a compact Hausdorff space and $n \geq 1$.
\item\label{9223_RSHA_2}
Any pullback of the form 
\begin{align*}
& B \oplus_{C (X^{(0)}, \, M_{n} (D)), \, \ph, \, \rh} C (X, M_{n} (D))
\\
& \hspace*{3em} {\mbox{}}
  = \bset{ (b, f) \in B \oplus C (X, M_{n} (D)) \colon
        \ph (b) = f |_{X^{(0)}} }
\end{align*}
is in $\mathcal{R}$ whenever $B \in \mathcal{R}$, $X$ is compact Hausdorff,
$n \geq 1$, $X^{(0)} \subseteq X$ is closed (possibly empty),
$\ph \colon B \to C (X^{(0)}, \, M_{n} (D))$
is any unital $^*$-homomorphism
(the zero \hm{} if $X^{(0)}$ is empty), and
$\rho \colon C (X, M_{n} (D)) \to C (X^{(0)}, \, M_{n} (D))$ is the
restriction $^*$-homomorphism.
\end{enumerate}
\end{defn}

\begin{ntn}\label{RSHAMachinery}
As noted in \cite[Definition~3.3]{archey2020structure}, every recursive subhomogenous algebra over $D$ is of the form
\[
R_D \cong
  \left[ \cdots \left[ \left[ C_{0} \oplus_{C_{1}^{(0)}, \ph_1, \rh_1}
              C_{1} \right]
    \oplus_{C_{2}^{(0)}, \ph_2, \rh_2} \right] \cdots \right]
   \oplus_{C_{l}^{(0)}, \ph_l, \rh_l} C_{l},
\]
where $C_{k} = C (X_{k}, \, M_{n (k)} (D))$ and $C_{k}^{(0)} = C \bigl( X_{k}^{(0)}, \, M_{n (k)} (D) \bigr)$ for positive integers $n(k)$ and $X_{k}^{(0)} \subseteq X_{k}$ compact sets ($X_k^0$ possibly empty), and where $\rho_{k} \colon C_{k} \to C_{k}^{(0)}$ are the restriction maps and $\ph_k\colon C_{k-1}\rightarrow C_k^{(0)}$ is any unital $^*$-homomorphism. We call the number $l$ the \emph{length} of the decomposition.
\end{ntn}

\begin{prop}\label{DpureRSHA}
Let $D$ be a unital simple pure \ca. Let $R_D$ be a recursive subhomogeneous algebra over $D$ and let $R$ be any recursive homogeneous algebra over $\mathbb{C}$.
Then both $R_D$ and $R \otimes D$ are pure.
\end{prop}
\begin{proof}
   We first show that $R_D$ is pure; an analogous proof applies to $R \otimes D$.
We argue by induction on the length of a recursive subhomogeneous decomposition
of $R_D$ over $D$.

The base case is when $R_D = C(X, M_{n}(D))$. 
Since $M_{n}(D)$ is again simple and pure, 
Corollary~\ref{Purehomogenous} implies that 
$C(X, M_{n}(D))$ is pure. For the inductive step, we may assume that
there is  unital pure \ca~$R_D'$,
$m \in \N$, a compact Hausdorff space~$X$,
a closed subset $X^{(0)} \subseteq X$,
and a $^*$-homomorphism
$\ph \colon R_D' \to C \bigl( X^{(0)}, \, M_{m} (D) \bigr)$
such that $R_D$ is the pullback defined as
\[
R_D = \bset{ (a, f) \in R_D' \oplus C (X, M_{m} (D))
   \colon \ph (a) = f |_{X^{(0)}}  }
\]
Since $f \mapsto f |_{X^{(0)}}$ is surjective and both
$R_D'$ and $C (X, M_{n} (D))$
are pure, it follows from Lemma~\ref{PurePullback}
that $R_D$ is pure.
This completes the proof. 

Next, we show that $R \otimes D$ is pure. As before, we proceed by induction on
the length of a decomposition of $R$.
 The base case is when $R= C(X)\otimes M_n$, which is pure by Corollary~\ref{Purehomogenous}. For the inductive step, write 
     \[
R = \bset{ (a, f) \in R' \oplus C (X, M_{m})
   \colon \ph (a) = f |_{X^{(0)}}  },
\]
where $R'$ is a recursive subhomogenous algebra such that $R'\otimes A$ is pure, $m \in \N$, $X^{(0)} \subset X$,
and 
$\ph \colon R' \to C \bigl( X^{(0)}, \, M_{m}) \bigr)$ is a $^*$-homorphism. By \cite[Theorem~3.8]{pedersen1999pullback}, $R\otimes A$ can be expressed as a pullback
\[
R'\otimes A \oplus_{ C(X^{(0)},M_m(A))} C(X, M_m(A)),
\]
with structure maps
\[
\ph\otimes \mathrm{id}_A
\quad\text{and}\quad
f|_{X^{(0)}}\otimes \mathrm{id}_A.
\]
By Lemma~\ref{PurePullback}, we conclude that $R\otimes A$ is pure, hence completing the induction step.
\end{proof}

Let $X$ be a compact Hausdorff space, and let $D$ be a simple unital \ca{}.
Let $A$ be a locally trivial $C(X)$-algebra with fiber $D$. Then, by repeated use of
Lemma 2.4 of \cite{dadarlat2009continuous}, we see that A is a recursive subhomogeneous algebra over $D$. Hence the following result follows from Proposition~\ref{DpureRSHA}.
\begin{cor}
    Let $D$ be a simple unital pure \ca{}. Let $X$ be a compact Hausdorff space and let $A$ be a locally trivial $C(X)$-algebra with fiber $D$. Then, $A$ is pure.
\end{cor}

The following is another implication of Proposition~\ref{DpureRSHA}.  Since any unital separable ASH-algebra can be written as an inductive limit of recursive subhomogeneous algebras, and limits of pure \ca{s} are pure \cite[Theorem~3.8]{perera2025extensions}, we conclude that

\begin{theorem}\label{prp:ASHtenPure}
   Let $B$ be a unital separable ASH-algebra and $A$ be simple and pure. Then, $B\otimes A$ is pure.
\end{theorem}

As a concrete example, pure \ca{s} can be obtaind by tensoring any Villadsen algebra of the first kind \cite{Vil99SRSimpleCa} with any reduced group \ca{} of a free group \cite{AGKEP25}.
   
\begin{cor}
    Let $B$ be a Villadsen algebra of the first kind. Then $B\otimes C_r^*(\mathbb{F}_2)$ is pure.
\end{cor}

We end this section with the following question:
\begin{qst}\label{qst:VillZstab}
    Let $B$ be a Villadsen algebra of the first kind without strict comparison (for example, Toms' \cite{Tom08ClassificationNuclear}). Is $B\otimes C_r^*(\mathbb{F}_2)$ $\mathcal{Z}$-stable?
\end{qst}

\begin{rem}\label{rem:Final}
    The analogue of Theorem \ref{prp:ASHtenPure} for pure, residually stably finite \ca{s} holds by using Theorem \ref{prp:MainNonSimp} instead of Corollary \ref{Purehomogenous}.
\end{rem}

\end{document}